\documentclass[a4paper,12pt,oneside]{article}

\usepackage[utf8]{inputenc}
\usepackage[english]{babel}
\usepackage{amsthm}
\usepackage{amssymb}
\usepackage{amsmath}
\usepackage{hyperref}
\usepackage{mathrsfs}

\usepackage{pstricks}

 \usepackage{tikz}

\usepackage{graphicx}
\graphicspath{{images/}}

\usepackage{longtable,geometry}
\usepackage{color}

\usepackage{hyperref}

\newtheoremstyle{note}{12pt}{12pt}{}{}{\bfseries}{.}{.5em}{}
\title{\LARGE\textbf{Quasi-Symmetric Conjugacy for Circle Maps with a Flat Interval}}
\author{Liviana Palmisano\\ Institute of Mathematics of PAN\\
Warsaw, Poland}

\makeatletter
\newtheorem{theo}[equation]{Theorem}
\newtheorem{prop}[equation]{Proposition}

\newtheorem{defin}[equation]{Definition}

\newtheorem{rem}[equation]{Remark}
\newtheorem{cor}[equation]{Corollary}

\numberwithin{equation}{section}
\newtheorem{lem}[equation]{Lemma}

\usepackage{babel}

\newcommand{\N}{{\mathbb N}}
\newcommand{\Z}{{\mathbb Z}}
\newcommand{\R}{{\mathbb R}}

\renewcommand{\S}{{\mathbb S}^1}

\newcommand{\Cd}{{{\mathcal C}^2}}

\newcommand{\C}{{\mathcal C}}

\newcommand{\F}{{\mathcal F}}
\newcommand{\G}{{\mathcal G}}

\renewcommand{\L}{{\mathscr L}}
\newcommand{\Lsec}{{\mathscr B}}

\newcommand{\Cr}{\operatorname{\mathbf{Cr}}}

\begin{document}
\maketitle
\author
\textcolor{blue}{}\global\long\def\sbr#1{\left[#1\right] }
\textcolor{blue}{}\global\long\def\cbr#1{\left\{  #1\right\}  }
\textcolor{blue}{}\global\long\def\rbr#1{\left(#1\right)}
\textcolor{blue}{}\global\long\def\ev#1{\mathbb{E}{#1}}
\textcolor{blue}{}\global\long\def\R{\mathbb{R}}
\textcolor{blue}{}\global\long\def\E{\mathbb{E}}
\textcolor{blue}{}\global\long\def\norm#1#2#3{\Vert#1\Vert_{#2}^{#3}}
\textcolor{blue}{}\global\long\def\pr#1{\mathbb{P}\rbr{#1}}
\textcolor{blue}{}\global\long\def\qq{\mathbb{Q}}
\textcolor{blue}{}\global\long\def\aa{\mathbb{A}}
\textcolor{blue}{}\global\long\def\ind#1{1_{#1}}
\textcolor{blue}{}\global\long\def\pp{\mathbb{P}}
\textcolor{blue}{}\global\long\def\cleq{\lesssim}
\textcolor{blue}{}\global\long\def\ceq{\eqsim}
\textcolor{blue}{}\global\long\def\Var#1{\text{Var}(#1)}
\textcolor{blue}{}\global\long\def\TDD#1{{\color{red}To\, Do(#1)}}
\textcolor{blue}{}\global\long\def\dd#1{\textnormal{d}#1}
\textcolor{blue}{}\global\long\def\eqdef{:=}
\textcolor{blue}{}\global\long\def\ddp#1#2{\left\langle #1,#2\right\rangle }
\textcolor{blue}{}\global\long\def\En{\mathcal{E}_{n}}
\textcolor{blue}{}\global\long\def\Z{\mathbb{Z}}
\textcolor{blue}{{} }

\textcolor{blue}{}\global\long\def\nC#1{\newconstant{#1}}
\textcolor{blue}{}\global\long\def\C#1{\useconstant{#1}}
\textcolor{blue}{}\global\long\def\nC#1{\newconstant{#1}\text{nC}_{#1}}
\textcolor{blue}{}\global\long\def\C#1{C_{#1}}
\textcolor{blue}{}\global\long\def\meas{\mathcal{M}}
\textcolor{blue}{}\global\long\def\cSpace{\mathcal{C}}
\textcolor{blue}{}\global\long\def\pspace{\mathcal{P}}

\begin{abstract}
In this paper we study quasi-symmetric conjugations of $\Cd$ weakly order-preserving circle maps with a flat interval.
Under the assumption that the maps have the same rotation number of bounded type and that bounded geometry holds we construct a quasi-symmetric conjugation between their non-wandering sets. Further, this conjugation is extended to a quasi-symmetric circle homeomorphism. Our proof techniques hinge on real-dynamic methods allowing us to construct the conjugation under general and natural assumptions.

\end{abstract}

\section{Introduction}
We contribute to the area of the understanding the dynamics of a class of continuous degree one circle maps with a flat interval, 
see \cite{VT1}, \cite{VT2}, \cite{MvSdMM}, \cite{Morgan}, \cite{5aut}, \cite{Grfun}, \cite{my}, \cite{my2}. 
In this paper we address the problem of quasi-symmetric conjugation. 
This field has been very active over the years. The reason is the close connection with two very important questions in one-dimensional dynamical systems: the renormalization conjecture (see \cite{Sull}) and the density of hyperbolicity (see \cite{Dencity}, \cite{DencityL}, \cite{Kozshenvan}, \cite{LevinVanStrien}). 
Various constructions of such conjugations have been proposed in the literature, see \cite{YoQS}, \cite{Jakobson}, \cite{Jak-Sw}, \cite{LevinVanStrien}, \cite{deFaria}, \cite{deMeloFaria}. 

The novelty of the paper stems from the fact that we employ only classical techniques coming from real dynamics. As a result we are able to reduce the hypotheses on the smoothness of the maps and on the degree of the criticality at the boundary points of the flat interval. 

Before stating our results formally we need to introduce basic assumptions and definitions. 
\subsection{Assumptions and Notations}\label{class}
We fix the clockwise orientation of the circle $\S=\R/\Z$ and use the standard metric.

A continuous circle map $f$ can be lifted using the universal cover to a continuous real map $F$ obeying $F(x + 1) -F(x) \in\Z$ for every $x\in \R$.  It will often be convenient to identify $f$ with one of its lifts $F$ and subsets of $\mathbb S^1$ with the corresponding subsets of $\mathbb R$.

\paragraph{Hypotheses.}
\begin{enumerate}
\item We consider continuous circle endomorphisms $f$ of degree one which have first derivative positive everywhere except for the closure of an open non-degenerate 
interval $U_f$ (the flat interval) on which it is equal to zero.

\item $f$ is at least twice continuously differentiable except at two points (endpoints of the flat interval).
\item Let $\left(a,b\right)$ be a preimage of $U_f$ under the natural projection of the real line on $\S$.
On some right-sided neighborhood of $b$, $f$ can be represented as
\begin{displaymath}
h_{r}\left(\left(x-b\right)^{\ell}\right),
\end{displaymath}
where $h_{r}$ is a $C^{2}$-diffeomorphism on an open neighborhood of $b$.
Analogously, there exists a $C^{2}$-diffeomorphism on a left-sided neighborhood of $a$ such that $f$ is of the form
\begin{displaymath}
h_{l}\left(\left(a-x\right)^{\ell}\right).
\end{displaymath}

\end{enumerate}
The real positive number $\ell$ is called the critical exponent of $f$.
 The class of such maps will be denoted by $\L$.


\paragraph{Standing Assumptions.} 
 Taken a map $f\in\L$ we will introduce a simplified notation for backward and forward images of the flat interval $U_f$. Instead of $f^{i}(U_f)$ we will  write $f^{i}$; for example, $f^{0}$ will denote the flat interval $U_f$.
 \paragraph {Distance between Points.}
 
We denote by $(a,b)=(b,a)$ the  shortest  open interval between $a$ and $b$ regardless of the order of these two points. 
The length of that interval denoted by $|a,b|$ is the natural metric on the circle. We adopt these notational conventions:
\begin{itemize}
\item $|f^{-i}|$ stands for the length of the interval $f^{-i}$.
\item  The left and the right endpoint of $f^{-i}$ are denoted by $l(f^{-i})$ and $r(f^{-i})$ respectively.
\item Consider $x\notin f^{-i}$. Then the distance from $x$ to the closer endpoint of $f^{-i}$ will be denoted by $|x,f^{-i}|$, and the distance to the more distant endpoint will be written explicitly with $|x, r(f^{-i})|$ or $|l(f^{-i}),x|$.
\item We define the distance between the endpoints of two intervals $f^{-i}$ and $f^{-j}$ analogously. For example, $|f^{-i},f^{-j}|$ denotes the distance between the closest endpoints of these two intervals while $|l(f^{-i}), f^{-j}|$ stands for $|f^{-i}|+|f^{-i},f^{-j}|$.
\end{itemize}

%

%

\subsection{Basic Definitions}\label{BD}
\paragraph{Rotation Number}
The maps which are continuous and weakly order preserving can be assigned the rotation number. If $f\in \L$  and $F$ is a lift of $f$ then the rotation number of $f$ is the limit
\begin{displaymath}
\rho(f)=\lim_{n\to\infty}\frac{F^n(x)}{n} \textrm{ (mod $1$)}.
\end{displaymath}
This limit exists and its value is independent of $x$. We exclude the trivial case of maps with rational rotation numbers and in the rest of this paper we will assume that the rotation number is irrational. Then it can be written as an infinite continued fraction
\begin{displaymath}
\rho(f)=\frac{1}{a_1+\frac{1}{a_2+\frac{1}{\cdots}}},
\end{displaymath}
where $a_i$ are positive integers.

If we cut off the portion of the continued fraction beyond the $n$-th position, and write the resulting fraction in lowest terms as $\frac{p_n}{q_n}$ then the numbers $q_n$ for $n\geq 1$ satisfy the recurrence relation
\begin{equation}\label{q_n}
q_{n+1} = a_{n+1}q_n+q_{n-1};\textrm{  } q_0 = 1;\textrm{  } q_1 = a_1.
\end{equation}

The number $q_n$ is the number of times we have to iterate the rotation by $\rho(f)$ in order that the orbit of any point makes its closest return so far to the point itself (see Chapter I, Sect. I in \cite{deMvS}).

\begin{defin}
We say that the rotation number $\rho(f)$ is of bounded type if there exists a constant $M>0$ such that, for all $n\in\N$ we have $a_n< M$.
\end{defin}
\paragraph{Scaling near Critical Point}
We define a sequence of scalings
\begin{displaymath}
\tau_{n}=\frac{\left|f^0,f^{q_n}\right|}{\left|f^0,f^{q_ {n-2}}\right|}. 
\end{displaymath}
These quantities measure `the geometry' near to the boundary points of the flat interval.
When $\tau_n\rightarrow 0$ we say that the geometry of the mapping is `degenerate'.
When $\tau_n$ is bounded away from zero we say that the geometry is `bounded'.

%

\paragraph{Non-wandering sets}
Let $f\in\L$ with flat interval $U_f$. The non-wandering set $K_f$ of $f$ is defined as
the set of the points $x$ such that for any open neighborhood $V\ni x$ there exists an integer $n>0$ such that the intersection of $V$ and $f^n(V)$ is non-empty.
By \cite{5aut}, $K_f=\S\setminus\cup_{i=0}^{\infty}f^{-i}$.

\paragraph{Quasi-symmetric homeomorphisms}

\begin{defin}
A homeomorphism $\varphi:\S\mapsto\S$ is called $Q$-quasi-symmetric, $Q>0$ if and only if for any two
intervals $I$ and $J$ of $\S$ with a common end-point such that $|I|= |J|$, 
\[
\frac{|\varphi(I)|}{|\varphi(J)|}\leq Q.
\]
\end{defin}
The same definition apply to the more general case of homeomorphisms of metric spaces. For more details see \cite{Heinonen}.
\subsection{Discussion and Statement of the Results}
As already mentioned in the introduction there are various reasons for studying quasi-symmetric conjugations between circle maps. 

The usual line of attack is by employing complex methods and constructing a quasi-conformal conjugacy of  holomorphic extensions of the considered maps 
(see \cite{YoQS}, \cite{Jakobson}, \cite{Jak-Sw}, \cite{LevinVanStrien}). While impressive, these techniques demand making assumptions which are not relevant to the problem, 
usually the maps are required to be analytic  and their orders of the criticality to be integers and the same for both maps 
(with a notable exception of \cite{VanStrienTrevor} pushing the methods to a more general setting).

The objective of this paper is to reduce assumptions as far as possible, which is achieved by utilizing the techniques of real dynamics. 
The real methods are typically much more technically demanding  (see for example \cite{Jakobson} and \cite{deMeloFaria}) compared to the complex ones due to the lack of the 
holomorphic structure. Many tools and properties useful in the complex setting are simply not valid anymore. For example the so-called ``gluing'' lemma, which is often a very efficient tool for gluing two quasi-conformal functions is not valid in the real line (the union of $f_1(x)=0$ for $x\leq 0$ and $f_1(x)=x^2$ for $x\geq 0$ is not quasi-symmetric in $0$).

We study the subclass $\Lsec$ of $\L$ consisting of  functions with 
irrational rotation number of bounded type and for which bounded geometry holds. Our main result is 

\begin{theo}\label{maintheo}
Let $f,g$ be in $\Lsec$ having the same rotation number $\rho$ and with the non-wandering sets $K_f$ and $K_g$. 
Then there exists a quasi-symmetric circle homeomorphism  $\varphi:\S\to\S$ which conjugates $f$ and $g$ on their non-wandering sets. i.e. 
\[
	\varphi \circ f_{| K_f} = g \circ \varphi_{| K_f}. 
\]
\end{theo}

By the second claim of Theorem 2 in \cite{5aut} any map with the irrational rotation number of bounded type and the critical exponent $\ell>2$ belongs to $\Lsec$.

The strength of our result is that we are able to construct the conjugacy under weak assumptions. We require the maps to be only $\Cd$ except for the boundary points of the flat interval. Moreover, we do not put any assumptions on the critical exponents, in particular they can be real and different for both the maps. Our work is the first contribution towards understanding the quasi-symmetric  conjugacy of maps with flat intervals. We present an extension of the conjugacy between the non-wandering sets on the whole circle. We stress that this is not straightforward due to lack of dynamical structure on the complement of non-wandering sets. We hope that our methods will be useful in construction of a global quasi-symmetric conjugacy of $f$ and $g$. We have not been able to do this yet. The difficulty is caused by unexpected distortion properties of the transition maps which require deeper understanding, see Section \ref{noQS} for more details.

As a part of the proof we develop a new technique of extending quasi-symmetric homomorphisms of Cantor sets to the whole circle. We use it for $K_f$ and $K_g$ but it holds in a much greater generality and is of independent interest. 
As we need much more notation we defer its presentation to the Appendix.
%

\section{Technical Tools}

\subsection{Cross-ratio}\label{distortion}
In order to control the distortion of iterates of maps in $\L$ we will use the cross-ratio $\Cr$.
\begin{defin}
If $a<b<c<d$ are four points on the circle, then we can define their \it{cross-ratio} $\Cr$ by:
\begin{equation*}
\Cr(a,b,c,d):=\frac{|b-a||d-c|}{|c-a||d-b|}.
\end{equation*}

\end{defin}

Here, we formulate the result which enables us to control the growth of the iterates of cross-ratios $\Cr$. 
The reader can refer to \cite{Swiatek} for the general case and to \cite{Grfun} for our situation.
\par
Consider a chain of quadruples
\begin{equation*}
\bigcup_{i = 0}^n \lbrace{(a_i,b_i,c_i,d_i)\rbrace} 
\end{equation*}
such that each is mapped onto the next by the map $f$. If the following conditions hold:
\begin{itemize}
\item There exists an integer $k \in \N$, such that each point of the circle belongs to at most $k$ of the intervals $(a_i,d_i)$.
\item The intervals $(b_i,c_i)$ do not intersect $f^0$.
\end{itemize}
Then, there exists a constant $K > 0$, independent of the set of quadruples, such that:
\begin{equation*}
\log\frac{\Cr(a_n, b_n, c_n, d_n)}{\Cr(a_0, b_0, c_0, d_0)}\leq K
\end{equation*}

\subsection {Continued Fractions and Partitions}\label{dynamicalpartition}
Let $f\in\L$. Since $f$ is order-preserving and has no periodic points, there exists an order-preserving and continuous map $h:\mathbb{S}^{1}\rightarrow\mathbb{S}^{1}$ such that $h\circ f=R_{\rho}\circ h$, where $\rho$ is the rotation number of $f$ and $R_{\rho}$ is the rotation by $\rho$. In particular, the order of points in an orbit of $f$ is the same as the order of points in an orbit of $R_{\rho}$. Therefore, results about $R_{\rho}$ can be translated into results about $f$, via the semiconjugacy $h$.

We define the so called dynamical partitions $\F_ {n}$ of $\mathbb S^1$ to study the geometric properties of $f$, see \cite{Grfun}. The partition $\F_{n}$ is generated by the first $q_{n}+q_{n+1}$ preimages of the flat interval and consists of
\begin{displaymath}
\left\{ f^{-i}: 0\leq i\leq q_{n+1}+q_{n}-1\right\},
\end{displaymath}
together with the gaps between these intervals.

We distinguish two kinds of gaps:
\begin{itemize}
\item The `long' gaps which are of the form
\begin{displaymath}
F^{i}_{ n}:=f^{-i}(F_{ n}^{0}), i=0,1,\ldots, q_{n+1}-1
\end{displaymath}
where $F_{ n}^{0}$ is the interval between $ f^{-q_n}$ and $f^{0}$ for $n$ even or the interval between $f^{0}$ and $f^{-q_n}$ for $n$ odd.
\end{itemize}
\begin{itemize}
\item The `short' gaps which are of the form
\begin{displaymath}
F^{i}_{ n+1}:=f^{-i}(F_{ n+1}^{0}), i=0,1,\ldots, q_{n}-1
\end{displaymath}
where $F_{ n+1}^{0}$ is the interval between $f^{0}$ and $f^{-q_{n+1}}$ for $n$ even or the interval between $f^{-q_{n+1}}$ and $f^{0}$ for $n$ odd.
\end{itemize}
We briefly explain the structure of the partitions. Take two consecutive dynamical partitions of order $n$ and $n+1$. 
The latter is a refinement of the former. All `short' gaps of $\F_{ n}$ become `long' gaps of $\F_{ n+1}$ while all `long' gaps of $\F_{ n}$ split into $a_{n+2}$ preimages of $f^{0}_{n}$ and $a_{n+2}$ `long' gaps and one `short' gaps of the next partition $\F_{ n+1}$:
\begin{equation}\label{div}
F_{ n}^{i}=\bigcup_{j=1}^{a_{n+2}}f^{-i-q_n-jq_{n+1}}\cup\bigcup_{j=0}^{a_{n+2}-1}F^{i+q_n+jq_{n+1}}_{ n+1}\cup F^{i}_{ n+2}.
\end{equation}


\begin{prop}\label{figo1}
Let $f\in\L$, then there exists a constant $C > 0$, such that for all $n \in \N$, if
$f^{-j}$ is a preimage belonging to the dynamical partition $\F_{n}$ and $F^{j}$ is one of the two gaps adjacent to $f^{-j}$, then:
$$\frac{\left|f^{-j}\right|}{\left|F^{j}\right|}\geq C$$
\end{prop}
\begin{cor}\label{cor:trouzero}
The lengths of gaps of the dynamical partition $\F_{n}$ tend to zero at least exponentially fast when $n\to\infty$.
\end{cor}
The proofs of Proposition \ref{figo1} and Corollary \ref{cor:trouzero} can be found in \cite{5aut}, pag. 606-607.

\section{Proof of Theorem \ref{maintheo}}

 From now on, we will always consider two functions $f,g\in\Lsec$ with the same irrational rotation number. 
This implies that, for all $n\in\N$, the partition generated by the dynamics of $f$, $\F_{ n}$ and the dynamical partition generated by $g$, 
$\G_{ n}$ have the same combinatorial structure. So, if $f^{-i}$ is a preimage of $\F_{ n}$, then the corresponding preimage of $\G_{ n}$ will be denoted by $g^{-i}$. 
In the same way if $F^i$ is a gap of $\F_{ n}$, then $G^i$ will denote the corresponding gap of $\G_{ n}$.

The proof of Theorem \ref{maintheo} consists of $3$ steps contained in separated subsections. These are:
\begin{enumerate}
\item construction of a conjugacy $\varphi_{0}$ of $f$ and $g$ on their non-wandering sets,
\item extension of $\varphi_{0}$ to a circle function $\varphi$,
\item proof of the fact that $\varphi$ is quasi-symmetric.
\end{enumerate}

\subsection{Conjugacy on the non-wandering sets}
In this section we construct a conjugacy between two functions $f$ and $g$ in $\Lsec$ on their non-wandering sets. This is formalized in the following proposition.
\begin{prop}\label{Cantorhomeo}
Let $f$ and $g$ be two functions in $\Lsec$ with the same irrational rotation number $\rho$ and with the non-wandering sets  $K_f$ and $K_g$. Then there exists a homeomorphism $\varphi_{0}:K_f\mapsto K_g$ such that $\varphi_{0}\circ f_{|K_f}=g\circ\varphi_{0}$.

\end{prop}

\begin{proof}

Notice that $K_f$ and $K_g$ are Cantor sets, so the construction of an homeomorphism between them is classic. 
Moreover we will give a sketch of the proof in order to familiarize with the homeomorphism $\varphi_{0}$ which will be often used in the subsequent parts of the paper. 

By Subsection \ref{BD} and Subsection \ref{dynamicalpartition}, $K_f$ and $K_g$ can be written as 
\[
K_f=\bigcap_{n\geq1}\left(\bigcup_{0\leq i_n\leq q_{n+1}-1}F^{i_n}\cup\bigcup_{0\leq j_n\leq q_{n}-1} F^{j_n}\right)
\]
and 
\[
K_g=\bigcap_{n\geq1}\left(\bigcup_{0\leq i_n\leq q_{n+1}-1}G^{i_n}\cup\bigcup_{0\leq j_n\leq q_{n}-1} G^{j_n}\right).
\]

Clearly, any point of $K_f$ and $K_g$ is uniquely encoded by a sequence $\cbr{k_{n}}_{n\geq1}$ such that for any $n$, 
$0\leq k_n\leq q_{n+1}-1$ or $0\leq k_n\leq q_{n}-1$.

\par
We define a mapping $\varphi_{0}:K_f\mapsto K_g$
by setting 
\[
\varphi_{0}(x):=\bigcap_{n\geq 1}G^{k_n},
\]
where $\cbr{k_{n}}_{n\geq0}$ encodes $x$ in the set $K_f$.
\par
Since $f$ and $g$ have the same rotation number, they share the same sequence of first return time $\{q_n\}_{n\in\N}$, thus there is a bijection between the set of sequences encoding points of $K_f$ with the set of sequences encoding points of $K_g$. Consequently $\varphi_{0}$ is a bijection. 
\par
In order to prove that $\varphi_{0}$ is continuous (the continuity of $\varphi^{-1}_{0}$ follows
by exchanging $f$ and $g$) we fix a sequence $\cbr{x_{n}}_{n\geq0}$
such that $x_{n}\in K_f$ and $x_{n}\to x\in K_f$. By Corollary \ref{cor:trouzero}, for any $\epsilon>0$ we may 
choose $m\in\N$ such that the length of any gap of the dynamical partition $\G_{m}$ is smaller than $\epsilon/2$. 
Moreover, we know that there exists a unique gap $F^{k_m}$ of the partition $\F_{m}$ such that $x\in F^{k_m}$. 
We recall that $K_f$ is a Cantor set (every point is a limit point) and that the convergence $x_{n}\to x$ holds in the Euclidean topology. 
So, there exists $n_{0}\in\N$ such that for any $n\geq n_{0}$ we have
$x_{n}\in F^{k_{m}}$ and in particular, by definition of $\varphi_{0}$, both $\varphi_{0}(x_{n})$ and $\varphi_{0}(x)$ are in $G^{k_{m}}$.
Finally for all $n\geq n_0$ $|\varphi_{0}(x_{n})-\varphi_{0}(x)|\leq\epsilon$, which concludes the proof.

By construction $\varphi_{0}$ is clearly a conjugacy of $f$ and $g$.

\end{proof}
\subsection{Extension of $\varphi_0$ to a global circle function}
We extend the function $\varphi_{0}$ constructed in Proposition \ref{Cantorhomeo} to all preimages of the flat interval.
The choices of the extension are restricted by our principal aim, which is to obtain a quasi-symmetric circle map.

Moreover,  while the function $\varphi_0$ was given by the dynamics, the construction of its extension is more delicate because of the lack of a dynamical structure on the preimages. This is one of the main reasons for which this case needs to be treated differently that the case of circle maps with one critical point. 

Our main idea is to reflect the points of $K_f$ in the preimages in order to create inside a ``dynamics'' which will naturally define the required extension.
 
\paragraph{Notation.} In order to make the construction clearer from now on we will use the following notation: all preimages $f^{-j_n}$ and all gaps $F^{j_n}$ of the dynamical partition $\F_n$ will be notated by $f^{-j_n}_{\scriptscriptstyle (0)}$ and $F^{j_n}_{\scriptscriptstyle (0)}$ and the partition containing them $\F_{\scriptscriptstyle (0), n}$. In order to lighten the notation we will omit the generation index $n$ of the partition,  where it is not relevant for the construction.

The main result of this section is:
\begin{prop}\label{extension}
There exists a circle homeomorphism $\varphi:\S\to\S$ such that $\varphi_{|K_f}=\varphi_{0}$.

\end{prop}

\begin{proof}

Let us denote by $K_{0,f}$ and $K_{0,g}$ the non-wandering sets of $f$ and $g$ respectively. We recall that $K_{0,f}=\S\setminus\bigcup_{i\geq 0} f^{-i}_{\scriptscriptstyle(0)}$ and $K_{0,g}=\S\setminus\bigcup_{i\geq 0} g^{-i}_{\scriptscriptstyle(0)}$.

We construct by induction two sequences of subintervals of $\S$
\[
 \{ \{f_{\scriptscriptstyle(i)}^{-j}\}_{j>0}\}_{i\geq1} \text{  and  }  \{ \{g_{\scriptscriptstyle(i)}^{-j}\}_{j>0}\}_{i\geq1},
\]

two sequences of subsets of $\S$, 
\[
 \{K_{i,f}\}_{i\geq1} \text{  and  } \{K_{i,g}\}_{i\geq1} 
\]
and a sequence of functions
\[
\{\varphi_{i}:K_{i,f}\to K_{i,g}\}_{i\geq1}
\]

which satisfy the following properties:

\begin{enumerate}
\item for all $j>0$ we consider the subintervals $f_{\scriptscriptstyle (i-1)}^{-j}$ and $g_{\scriptscriptstyle(i-1)}^{-j}$ of $\S$ generated in the step $i-1$. 
Then there exist $k'_i(j), k_i(j)>0$ and intervals $f_{\scriptscriptstyle(i)}^{-l}\subset f_{\scriptscriptstyle (i-1)}^{-j}$ 
(resp. $g_{\scriptscriptstyle(i)}^{-l}\subset g_{\scriptscriptstyle (i-1)}^{-j}$) with $l\geq\min\{k_i(j), k'_i(j)\}$ such that $f_{\scriptscriptstyle(i)}^{-l}$ 
(resp. $g_{\scriptscriptstyle(i)}^{-l}$) 
is comparable with $f_{\scriptscriptstyle(i-1)}^{-l}\in(l(f_{\scriptscriptstyle(i-1)}^{-k_i(j)}), f_{\scriptscriptstyle(i-1)}^{-j})
\cup(f_{\scriptscriptstyle(i-1)}^{-j},r(f_{\scriptscriptstyle(i-1)}^{-k'_i(j)}))$ (resp. $g_{\scriptscriptstyle(i-1)}^{-l}
\in(l(g_{\scriptscriptstyle(i-1)}^{-k_i(j)}), g_{\scriptscriptstyle(i-1)}^{-j})\cup(g_{\scriptscriptstyle(i-1)}^{-j},r(g_{\scriptscriptstyle(i-1)}^{-k'_i(j)}))$ ) 

\item $K_{i,f}=K_{i-1,f}\cup\bigcup_{j>0} (f^{-j}_{\scriptscriptstyle(i-1)}\setminus\bigcup_{l\geq\min\{k_i(j), k'_i(j)\}} {f^{-l}_{\scriptscriptstyle(i)}})$ 
(resp. $K_{i,g}=K_{i-1,g}\cup\bigcup_{j>0} (g^{-j}_{\scriptscriptstyle(i-1)}\setminus\bigcup_{l\geq\min\{k_i(j), k'_i(j)\}} g^{-l}_{\scriptscriptstyle(i)})$).
 \item $\varphi_{{i}|_{K_{i-1,f}}}=\varphi_{i-1}$
\end{enumerate}

Let us start the construction for $i=1$.

We fix $j\geq 0$ and we consider the preimages $f^{-j}_{\scriptscriptstyle (0)}$ and $g^{-j}_{\scriptscriptstyle (0)}$ of the dynamical partitions $\F_m$ and $\G_m$. 
Let $k$ be the smallest positive integer such that,

\[
|l(f^{-j-q_{k}}_{\scriptscriptstyle (0)}),f^{-j}_{\scriptscriptstyle (0)}|+|f^{-j}_{\scriptscriptstyle (0)},r(f^{-j-q_{k+1}}_{\scriptscriptstyle (0)})|\leq |f^{-j}_{\scriptscriptstyle (0)}|
\]
and 
\[
|l(g^{-j-q_{k}}_{\scriptscriptstyle (0)}),g^{-j}_{\scriptscriptstyle (0)}|+|g^{-j}_{\scriptscriptstyle (0)},r(g^{-j-q_{k+1}}_{\scriptscriptstyle (0)})|\leq |g^{-j}_{\scriptscriptstyle (0)}|. 
\]
Notice that in the previous inequalities we suppose that $k$ is even. Due to the symmetry the case $k$ being odd is analogous.
For all $\alpha\geq 1$ we generate the following subsets of $f^{-j}_{\scriptscriptstyle (0)}$

\[
R^j_{\alpha, f}=\left\{-\alpha(x-r(f^{-j}_{\scriptscriptstyle (0)}))+r(f^{-j}_{\scriptscriptstyle (0)}) | x\in K_{0,f} \cap (f^{-j}_{\scriptscriptstyle (0)},r(f^{-j-q_{k+1}}_{\scriptscriptstyle (0)})) \right\}
\]
and
\[
L^{j}_{\alpha, f}=\left\{-\alpha(x-l(f^{-j}_{\scriptscriptstyle (0)}))+l(f^{-j}_{\scriptscriptstyle (0)}) |  x\in K_{0,f}\cap (l(f^{-j-q_{k}}_{\scriptscriptstyle (0)}),f^{-j}_{\scriptscriptstyle (0)})  \right\}
\]
Analogously we define the respectively two subsets of $g^{-j}_{\scriptscriptstyle (0)} $, $R^j_{\alpha, g}$ and $R^j_{\alpha, g}$.

Let $\alpha_1,\alpha'_1\geq 1$ such that either $L_{\alpha_1,f}^{j}\cap R_{\alpha_1 ,f}^{j}$ and $L_{\alpha'_1,g}^{j}\cap R_{\alpha'_1 ,g}^{j}$ is a single point. 

In order to simplify the notation we denote by $L_{1,f}^{j}=L_{\alpha_{1},f}^{j}$, $R_{1,f}^{j}=R_{\alpha'_{1},f}^{j}$, $L_{1,g}^{j}=L_{\alpha'_{1},g}^{j}$ and $R_{1,f}^{j}=R_{\alpha'_{1},f}^{j}$.

Observe that  $L_{1,f}^{j}=f^{-j}_{\scriptscriptstyle (0)}\setminus\bigcup_{l\geq j+q_{k}}f^{-l}_{\scriptscriptstyle(1)}$ where, for every $l$, 
$f^{-l}_{\scriptscriptstyle(1)}$ 
is comparable with the preimage 
$f^{-l}_{\scriptscriptstyle (0)}\subset  (l(f^{-j-q_{k}}_{\scriptscriptstyle (0)}),f^{-j}_{\scriptscriptstyle (0)})$ 
and $R_{1,f}^{j}=f^{-j}_{\scriptscriptstyle (0)}\setminus\bigcup_{l\geq j+q_{k+1}}f^{-l}_{\scriptscriptstyle(1)}$ with $f^{-l}_{\scriptscriptstyle(1)}$ 
comparable with the preimage $f^{-l}_{\scriptscriptstyle (0)}\subset  (f^{-j}_{\scriptscriptstyle (0)},r(f^{-j-q_{k+1}}_{\scriptscriptstyle (0)}))$. 
Let use denote $k_1(j):=j+q_k$ and $k'_1(j):=j+q_{k+1}$ . Repeating the construction for all preimages we define 

\[
K_{1,f}=K_{0,f}\cup\bigcup_{j\geq 0}  (L_{1,f}^{j}\cup R_{1,f}^{j})=\bigcup_{j\geq 0}(f^{-j}_{\scriptscriptstyle (0)}\setminus\bigcup_{l\geq\min\{k_1(j), k'_1(j)\}}f^{-l}_{\scriptscriptstyle(1)})
\]
and 
\[
K_{1,g}=K_{0,g}\cup\bigcup_{j\geq 0} (L_{1,g}^{j}\cup R_{1,g}^{j})=\bigcup_{j\geq 0}(g^{-j}\setminus\bigcup_{l_\geq\min\{k_1(j), k'_1(j)\}}g^{-l}_{\scriptscriptstyle(1)})
\] 
and $\varphi_1:K_{1,f}\to K_{1,g}$ as follows:
\begin{itemize}
 \item If $x\in K_{0,f}$ $\varphi_1(x)=\varphi_0(x)$,
 \item

If 
$x\in L_{1,f}^{j}$ then there exists $y\in K_{0,f}\cap (l(f^{-j-q_{k}}_{\scriptscriptstyle (0)}),f^{-j}_{\scriptscriptstyle (0)}) $ such that $x=-\alpha_1(y-l(f^{-j}_{\scriptscriptstyle (0)}))+l(f^{-j}_{\scriptscriptstyle (0)})$. So, $\varphi_1(x)=-\alpha'_1(\varphi_0(y)-l(g^{-j}_{\scriptscriptstyle (0)}))+l(g^{-j}_{\scriptscriptstyle (0)})\in L_{1,g}^{j}$.

\item 
If 
$x\in R^{1,j}_f$, as in the previous case, $x=-\alpha_1(y-r(f^{-j}_{\scriptscriptstyle (0)}))+r(f^{-j}_{\scriptscriptstyle (0)})$ with $y\in K_{0,f} \cap (f^{-j}_{\scriptscriptstyle (0)},r(f^{-j-q_{k+1}}_{\scriptscriptstyle (0)}))$. Then $\varphi_1(x)=-\alpha'_1(\varphi_0(y)-r(g^{-j}_{\scriptscriptstyle (0)}))+r(g^{-j}_{\scriptscriptstyle (0)})\in R_{1,g}^{j}$. 
\end{itemize}

Suppose now that we have already constructed $ \{f_{\scriptscriptstyle(n)}^{-j}\}_{j>0}$,  $ \{g_{\scriptscriptstyle(n)}^{-j}\}_{j>0}$, 
$K_{n,f}$, $K_{n,g}$, $\varphi_n$, and we construct $ \{f_{\scriptscriptstyle(n+1)}^{-j}\}_{j>0}$,  $ \{g_{\scriptscriptstyle(n+1)}^{-j}\}_{j>0}$, $K_{n+1,f}$, $K_{n+1,g}$, $\varphi_{n+1}$. Let us consider all the subintervals $f^{-j}_{\scriptscriptstyle(n)}$ generated in the previous step. By construction there exist $j'$ and $k'_n(j'), k_n(j')>0$ such that $f^{-j}_{\scriptscriptstyle(n)}\subset f^{-j'}_{\scriptscriptstyle(n-1)}$ and $j\geq k_n(j')$. Then we distinguish two cases:

 \begin{enumerate}
 \item $j=k_n(j')$ or $j=k'_n(j')$ 
  \item  $j\neq \{k_n(j'), k'_n(j')\}$.
 \end{enumerate}
 In the second one we repeat exactly the construction of the initialization case generating 
 $L_{n+1,f}^{j}\cup R_{n+1,f}^{j}=f^{-j}_{\scriptscriptstyle(n)}\setminus\bigcup_{l\geq \min\{k_{n+1}(j), k'_{n+1}(j)\}}f^{-l}_{\scriptscriptstyle(n+1)}$ and  
 $R_{n+1,g}^{j}\cup L_{n+1,g}^{j}=g^{-j}_{\scriptscriptstyle(n)}\setminus\bigcup_{l\geq \min\{k_{n+1}(j), k'_{n+1}(j)\}}g^{-l}_{\scriptscriptstyle(n+1)}$.
 
  Let us now consider a subinterval $f^{-j}_{\scriptscriptstyle(n)}$ which satisfies condition $1$. Because of symmetry we can suppose that 
  $f^{-j}_{\scriptscriptstyle (n)}\subset L_{n,f}^{j'}$ (the other case is analogous).
   Let $k$ be the smallest positive integer such that,

\[
|l(f^{-j-q_{k}}_{\scriptscriptstyle(n)}),f^{-j}_{\scriptscriptstyle(n)}|\leq |f^{-j}_{\scriptscriptstyle(n)}|
\]
and 
\[
|l(g^{-j-q_{k}}_{\scriptscriptstyle(n)}),g^{-j}_{\scriptscriptstyle(n)}|\leq |g^{-j}_{\scriptscriptstyle(n)}|. 
\]
For all $\alpha\geq 1$, we consider the subset
\[
CL^{j}_{\alpha,f}=\left\{-\alpha(x-l(f^{-j}_{\scriptscriptstyle (n)}))+l(f^{-j}_{\scriptscriptstyle (n)}) |  x\in K_{n,f}\cap (l(f^{-j-q_{k}}_{\scriptscriptstyle (n)}),f^{-j}_{\scriptscriptstyle (n)}) \right\}
\]
and in particular we can choose $\alpha_{n+1,c}\geq 1$ such that the set $CL_{n+1,f}^{j}=CL^{j}_{\alpha_{n+1,c},f}$ and $ f^{-j}_{\scriptscriptstyle(n)}$ intersect only in the point $r( f^{-j}_{\scriptscriptstyle(n)})$. In the same way we can generate $CL_{n+1,g}^{j}$. 

If $f^{-j}_{\scriptscriptstyle (n)}\subset R_{n,f}^{j'}$ in the same way we produce $CR_{n+1,f}^{j}$ and $CR_{n+1,g}^{j}$.

Finally, we define 
 \[
 K_{n+1,f}=K_{n,f}\cup\bigcup_{j}(L_{n+1,f}^{j}\cup R_{n+1,f}^{j})\cup\bigcup_{j} CL_{n+1,f}^{j} \cup\bigcup_{j} CR_{n+1,f}^{j}
 \]
 and
\[
  K_{n+1,g}=K_{n,g}\cup\bigcup_{j}(L_{n+1,g}^{j}\cup R_{n+1,g}^{j})\cup\bigcup_{j} CL_{n+1,g}^{j} \cup\bigcup_{j} CR_{n+1,g}^{j} 
 \]

and $\varphi_{n+1}:K_{n+1,f}\to K_{n+1,f}$ as follows. 
 Let $x\in K_{n+1,f}$.
\begin{itemize}
\item If $x\in K_{n,f}$, $\varphi_{n+1}(x)=\varphi_n(x)$.
 \item
If 
$x\in CL_{n+1,f}^{j}$ then $\varphi_{n+1}(x)=-\alpha'_{n+1,c}(\varphi_n(y)-l(g^{-j}_{\scriptscriptstyle (n)}))+l(g^{-j}_{\scriptscriptstyle (n)})\in L_{n+1,g}^{j}$ with $y\in K_{n,f}\cap (l(f^{-j-q_{k}}_{\scriptscriptstyle(n)}),f^{-j}_{\scriptscriptstyle(n)})$. 
\item 
If 
$x\in CR_{n+1,f}^j$, then $\varphi_{n+1}(x)=-\alpha'_{n+1,c}(\varphi_n(y)-r(g^{-j}_{\scriptscriptstyle (n)}))+r(g^{-j}_{\scriptscriptstyle (n)})\in R_{n+1,g}^{j}$ with $y\in K_{n,f}\cap ((f^{-j}_{\scriptscriptstyle(n)}, r(f^{-j-q_{k+1}}_{\scriptscriptstyle(n)}))$. 
\end{itemize}
Analogously if $x\in L_{n+1,f}^{j}$ or $x\in R_{n+1,f}^{j}$.

 We can now define $\varphi:=\bigcup\varphi_i:\bigcup K_{i,f}\to\bigcup K_{i,g}$. 
Observe that the length of the subintervals $f^{-j}_{\scriptscriptstyle(i)}$ tends to zero with $i$, so $\bigcup K_{i,f}=\S=\bigcup K_{i,g}$. 
By construction $\varphi$ is an order preserving bijection of the circle, therefore it is a circle homeomorphism.

\end{proof}
Let us make some observation concerning the proof of Proposition \ref{extension}. We also introduce notations for the new elements produced there.
\begin{rem}\label{howfillpreimages} 
 Let $f^{-i_n}_{\scriptscriptstyle(0)}$ be a preimage of the dynamical partition $\F_{\scriptscriptstyle(0),n}$. In the first step, 
 we reproduce in $f^{{-i_n}}_{\scriptscriptstyle(0)}$ the same structure of gaps and preimages presented in Subsection \ref{dynamicalpartition}. We will denote by $f^{-j_m}_{\scriptscriptstyle (1)}$ and $F^{j_m}_{\scriptscriptstyle (1)}$ the new ``preimages'' and ``gaps'' comparable with some preimage and gap of $\F_{\scriptscriptstyle(0),m}$. Abusing the vocabulary we will say that they belong to the ``dynamical partition'' $\F_{\scriptscriptstyle (1),m}$. 
 In the step $n$  all the preimages $f^{-l_k}_{\scriptscriptstyle (n)}$ of $\F_{\scriptscriptstyle (n),k}$ will then be subdivided generating new ``preimages''
$f^{-j_m}_{\scriptscriptstyle (n+1)}$ and ``gaps'' $F_{\scriptscriptstyle (n+1)}^{-j_m}$ comparable with elements of $\F_{\scriptscriptstyle (n), m}$ and belonging to the ``dynamical partition'' $\F_{\scriptscriptstyle (n+1), m}$.
\end{rem}

\begin{defin}\label{extreme}
 Let $f^{-j}_{\scriptscriptstyle(m)}$ be a preimage of $\F_{\scriptscriptstyle (m), n}$. 
 If there exists another preimage $f^{-j_1}_{\scriptscriptstyle (m)}$ of  $\F_{\scriptscriptstyle (m), n}$ 
 such that $r(f^{-j}_{\scriptscriptstyle (m)})=l(f^{-j_1}_{\scriptscriptstyle (m)})$ or $r(f^{-j_1}_{\scriptscriptstyle (m)})=l(f^{-j}_{\scriptscriptstyle (m)})$ 
 then we say that $f^{-j}_{\scriptscriptstyle(m)}$ is a right extreme or a left extreme respectively preimage. 
 \end{defin}
\begin{rem}\label{CP}
Let $m\in\N$ and let $f^{-i_n}_{\scriptscriptstyle(m)}$ be a right (resp. left) extreme  preimage of $\F_{\scriptscriptstyle (m), n}$ and let   $f^{-i_{n+k}}_{\scriptscriptstyle(m)}$ be the right (resp. left) extreme preimage of $\F_{\scriptscriptstyle (m+1), n+k}$ which is contained in $f^{-i_n}_{\scriptscriptstyle(m)}$. Then, by construction $f^{-i_n}_{\scriptscriptstyle(m)}$ and $f^{-i_{n+k}}_{\scriptscriptstyle(m)}$ are comparable.
\end{rem}

\subsection{$\varphi$ is quasi-symmetric}

We state a few lemmas which play a key role in proving that $\varphi$ is quasi-symmetric. 
From now on, when there is not possibility of confusion in denoting the preimages and the gaps of the dynamical partitions presented in Subsection \ref{dynamicalpartition} we will omit the subscript $(0)$.
\begin{lem}\label{cor:trouzero1}
For all $m$, the lengths of gaps of the dynamical partition $\F_{\scriptscriptstyle(m), n}$ tend to zero at least exponentially fast when $n\to\infty$.
\end{lem}
\begin{proof}
For $m=0$ we use Corollary \ref{cor:trouzero}. By Remark \ref{howfillpreimages} the claim follows for $m>0$. 
\end{proof}

\begin{lem}\label{bboungap}

Fix $f\in\Lsec$. Let $ f^{-j}$ be a preimage of the dynamical partition $\F_{ n}$ and let $F^{j}$ be one of the two gaps adjacent to $ f^{-j}$. If $ f^{-j}$ and  $F^{j}$ are contained in the same gap of the partition $\F_{ n-1}$, then $ f^{-j}$ and $F^{j}$ are comparable.
\end{lem}
\begin{proof}
See Lemma $5.1$ in \cite{my}. 
\end{proof}
\begin{lem}\label{boungap1}
For all $f\in\Lsec$, every two adjacent gaps of the dynamical partition $\F_{ n}$ are comparable.
\end{lem}
\begin{proof}
This is Proposition $3$ in \cite{my}. 
\end{proof}

\begin{lem}\label{boungap}
Let $f\in\Lsec$ and let $F^{j}$ be a gap of the dynamical partition $\F_{n}$. 
Then $F^{j}$ is comparable  with the gap of the partition $\F_{ n-1}$ which contains it.
\end{lem}
\begin{proof}
Due to the symmetry, 
we may assume that $n$ is even, that is $f^{-q_n}$ is to the left of $f^{0}$. 
Let $F_{ n-1}$ a gap of the partition $\F_{n-1}$. Without loss of generality we may suppose that 
\[
F_{ n-1}=(f^{0}, f^{-q_{n-1}}).
\]
So, the gaps of $\F_{ n}$ which are contained in $F_{ n-1}$ are of the form

\[
F^{i}_n=( f^{-q_{n+1}+iq_n}, f^{-q_{n+1}+(i+1)q_n})
\]
with $i\in\{0,\dots, a_{n+1}-1\}$.

Notice that in the previous representation for the gaps of $\F_{ n}$ contained in $F_{ n-1}$ we do not consider the first gap $(f^{0}, f^{q_{n+1}})$ which is clearly comparable with $F_{ n-1}$ 
by the bounded geometry hypothesis. 

Let $i\in\{0,\dots, a_{n+1}-1\}$, then
\[
\frac{|F^{i}_n|}{|F_{n-1}|}\geq C_1 \frac{|F^{(i-1)}_n|}{|F_{n-1}|}\geq C_2 \frac{|F^{(i-2)}_n|}{|F_{ n-1}|}\geq\dots\geq C_{i+1} \frac{|f^{0}, f^{q_{n+1}}|}{|F_{n-1}|}
\]
which is bounded away from zero, by the bounded geometry hypothesis. Notice that in the previous inequality we used $i+1$ times Lemma \ref{boungap1} and the hypotheses that the rotation number is of bounded type.
\end{proof}
\begin{lem}\label{bounpreim}
Let $f\in\Lsec$ and let $f^{-j}$ be a preimage of the dynamical partition $\F_{ n}$. If $f^{-j}$ is not a preimage of the dynamical partition $\F_{ n-1}$, then $f^{-j}$ is comparable  with the gap of $\F_{n-1}$ which contains it.
\end{lem}

\begin{proof} 
Let $F^j$ be a gap of the dynamical partition $\F_{ n}$ which is adjacent to $f^{-j}$. 
Then the prove follows by Lemma \ref{bboungap} and Lemma \ref{boungap}.
\end{proof}

\begin{lem}
\label{moveFromPreimageRight}

There exists $C>0$ such that
for any gap $F_{n}$ of the dynamical partition $\F_{n}$ and $x\in F_{n}$,
there exists $k\geq n$ and a gap $F_{k}$ of $\F_{ k}$,  such that:
\begin{itemize}
\item[-] $x\in F_{k}$, 
\item[-]$r(F_{k})=r(F_{n})$ (resp. $l(F_{k})=l(F_{n})$),
\item[-] $|F_{k}|\leq C|{x},r(F_{n})|$ 
(resp. $|F_{k}|\leq C|l(F_{n}),x|$).
\end{itemize}
\end{lem}
\begin{proof}
Let $k$ be the largest natural number such that a gap $F_{k}$ of $\F_{ k}$ satisfies the following conditions: $x\in F_{ k}$ and $r(F_{n})=r(F_{k})$.
Its existence is guaranteed by Lemma \ref{cor:trouzero1}. 
By the maximality of $k$ there exists a gap $F_{k+1}$ of $\F_{ k+1}$ 
which is contained in $(x, r(F_{n}))$.
Then, by Lemma \ref{boungap} there exists a uniform constant $C$ such that:
\[
 |F_{k}|\leq C |F_{k+1}|\leq C|{x},r(F_{n})|.
\]

\end{proof}

\begin{lem}\label{comparability}
Let $f\in\Lsec$. Then there exist two constants $C_1$ and $C_2$ such that, for any natural number $\alpha$ and for any preimage $f^{-i_{n+\alpha}}$ of the dynamical partition $\F_{ n+\alpha}$ which is contained in a gap $F_{n}$ of $\F_{ n}$ the following holds
\[
C_1^{\alpha}\leq\frac{|f^{-i_{n+\alpha}}|}{|F_{ n}|}\leq C_2^{\alpha}.
\]
\end{lem}

\begin{proof}
Observe that
\[
\frac{|f^{-i_{n+\alpha}}|}{|F_{n}|}=\frac{|f^{-i_{n+\alpha}}|}{|F_{n+\alpha-1}|}\frac{|F_{n+\alpha-1}|}{|F_{n+\alpha-2}|}\dots\frac{|F_{n+1}|}{|F_{n}|}.
\]
where, for all $k\in\{n+1,n+\alpha-1\}$, $F_{ k}$ is the gap of $\F_{ k}$ which contains $F_{ k+1}$. 
The above estimation is a consequence of Lemma \ref{boungap} and Lemma \ref{bounpreim}. 
\end{proof}

%
%
%
%

\begin{rem} \label{ppropertiesofconstraction}
Let $f^{-i_n}_{\scriptscriptstyle(m)}$ be a preimage of $\F_{\scriptscriptstyle (m), n}$ which separates two circle points $x$ and $y$. 
Then, by the construction of $\varphi$, $\varphi(x)$ and $\varphi(y)$ are separated by the corresponding preimage $g^{-i_n}_{\scriptscriptstyle(m)}$ of the dynamical partition $\G_{\scriptscriptstyle (m), n}$.
Obviously in the same way, 
if three points $x,y,z$ of the circle are contained in a same gap or preimage of $\F_{\scriptscriptstyle (m), n}$
then $\varphi(x),\varphi(y),\varphi(z)$ are contained in the corresponding gap or preimage of $\G_{\scriptscriptstyle (m), n}$.
\end{rem}


\begin{lem}\label{key}
 Let $f,g\in\Lsec$ having the same rotation number $\rho$, let $\varphi$ be the function constructed in Proposition \ref{extension} and let $x,y,z,w\in\S$. Assume that the following conditions are satisfied:
\begin{itemize}
\item  $x,y,z,w$  belong to a same gap $F_{n}$ of the partition $\F_{ n}$,
\item  for all $\alpha>0$, $y$ and $z$ are not contained in a some preimage of $\F_{n+\alpha}$,
\item for a positive constant $C$, $|y,z|\geq C |F_{n}|$.

\end{itemize} 
 
 Then there exists a constant $C_1>0$ such that  
 \[
\frac{|\varphi(x),\varphi(w)|}{|\varphi(y),\varphi(z)|}\leq C_1.
\]

\end{lem}
\begin{proof}
By Corollary \ref{cor:trouzero1}, Lemma \ref{bboungap} and Lemma \ref{comparability}, there exists $\alpha(C)>0$ and a premiage $f^{-i_{n+\alpha}}$ of the partition $\F_{ n+\alpha}$ which separates $y$ and $z$.
By Remark \ref{propertiesofconstraction}, $\varphi(x), \varphi(y)$ and $\varphi(z)$ are contained in $G_{ n}$ and $g^{-i_{n+\alpha}}\subset (\varphi(y),\varphi(z))$.
So,
\[
\frac{|\varphi(x),\varphi(w)|}{|\varphi(y),\varphi(z)|}\leq\frac{|G_{n}|}{|g^{-i_{n+\alpha}}|}\leq C_1,
\]
where the constant $C_1$ comes from Lemma \ref{comparability}. 
\end{proof}

%
%
\begin{rem}\label{comnew}
Observe that by the construction in the proof of Proposition \ref{extension} explained in Remark \ref{howfillpreimages}, Lemma \ref{bboungap}, Lemma \ref{boungap1}, Lemma \ref{boungap}, Lemma \ref{bounpreim}, Lemma \ref{moveFromPreimageRight}, Lemma \ref{comparability}, Lemma \ref{ppropertiesofconstraction} and Lemma \ref{key} are true also for the elements of the partition $\F_{\scriptscriptstyle (k)}$ for all finite natural number $k$.
\end{rem}
\begin{lem}\label{key1}
 Let $f,g\in\Lsec$ having the same rotation number $\rho$, let $\varphi$ be the function constructed in Proposition \ref{extension} and let $x,y,z,k\in\S$. 
 
 Assume that the following conditions are satisfied:
\begin{itemize}
\item  $x,y,z,w$  belong to a same gap $F_{n}$ of the partition $\F_{ n}$,
\item  there exists $\alpha>0$ and a preimage $f^{-i_{n+\alpha}}$ of $\F_{ n+\alpha}$ such that $(y,z)\subset f^{-i_{n+\alpha}}$,
\item for a positive constant $C$, $|y,z|\geq C |F_{n}|$.
\end{itemize}
 
 Then there exists a constant $C_1>0$ such that  
 \[
\frac{|\varphi(x),\varphi(w)|}{|\varphi(y),\varphi(z)|}\leq C_1
\]

\end{lem}
\begin{proof}
 By Lemma \ref{comparability}
\begin{equation}\label{forg}
|g^{-i_{n+\alpha}}|\geq C_1 |G_{n}|
\end{equation}
with a positive constant $C_1$. 

By the construction of $\varphi$,  $f^{-i_{n+\alpha}}$ 
is divided into four sub-sets $F_{\scriptscriptstyle (1)}^{j_{k}}$, $f_{\scriptscriptstyle (1)}^{-j_{k}}$, $f_{\scriptscriptstyle (1)}^{-j_{k+1}}$ and $F_{\scriptscriptstyle (1)}^{j_{k+1}}$ which are gaps and preimages of $\F_{\scriptscriptstyle (1),k}$ (see Remark  \ref{howfillpreimages}). If $(y,z)$ contains $f_{\scriptscriptstyle (1)}^{-j_{k}}$ or
$f_{\scriptscriptstyle (1)}^{-j_{k+1}}$, by Remark \ref{ppropertiesofconstraction}, (\ref{forg}) and  Remark \ref{howfillpreimages} the conclusion is immediate. In fact
\[
\frac{|\varphi(x),\varphi(w)|}{|\varphi(y),\varphi(z)|}\leq\frac{|G_{n}|}{|g_{\scriptscriptstyle (1)}^{-i_{k}}|}\leq C_2\frac{|G_{n}|}{|g^{i_{n+\alpha}}|}\leq C_3
\]

Otherwise the following four cases are possible:
\begin{enumerate}
\item $y\in f_{\scriptscriptstyle (1)}^{-j_{k}}$ and $z\in f_{\scriptscriptstyle (1)}^{-j_{k+1}}$. 
Suppose that $|y,r(f_{\scriptscriptstyle (1)}^{-j_{k}})|\geq \frac{1}{2}|y,z|$. 
Then,  $| f_{\scriptscriptstyle (1)}^{-j_{k}} |\leq|F_{n}|\leq 4 C|(y,r(f_{\scriptscriptstyle (1)}^{-j_{k}})|$. So, by Remark \ref{CP}  there exists a constant $\beta>0$ and a right extreme preimage  $f^{-i_{m}}_{\scriptscriptstyle(1+\beta)}$ of $\F_{\scriptscriptstyle (1+\beta), m}$, $m\geq k$ which is contained in $(y,r(f^{-j_{k}}_{\scriptscriptstyle(1)}))$. By Remark  \ref{CP}, Remark \ref{ppropertiesofconstraction} and (\ref{forg}),

\[
\frac{|\varphi(x),\varphi(w)|}{|\varphi(y),\varphi(z)|}\leq\frac{|G_{n}|}{|g^{-i_{m}}_{\scriptscriptstyle(1+\beta)}|}\leq C_4
\]

Because of symmetry the case $|x,r(f_{\scriptscriptstyle (1)}^{-j_{k}})|< \frac{1}{2}|x,y|$ is analogous.
\item $y,z\in F_{\scriptscriptstyle (1)}^{j_{k}}$ or $y,z\in F_{\scriptscriptstyle (1)}^{j_{k+1}}$. Observe that by the third point of our hypothesis,  up to a finite number of steps, we can assume that $(y,z)$ are not contained is a preimage of some next generation $\F_{\scriptscriptstyle (k)}$. So, by Lemma \ref{key} and Remark \ref{comnew}
\[
\frac{|\varphi(l(F_{\scriptscriptstyle (1)}^{j_{k}})),\varphi(r(F_{\scriptscriptstyle (1)}^{j_{k}}))|}{|\varphi(y),\varphi(z)|}\leq C_5
\]
We conclude using \label{propertiesofconstraction} Lemma \ref{bounpreim} and Remark \ref{comnew}.
\item $y\in F_{\scriptscriptstyle (1)}^{j_{k}}$ and $z\in f_{\scriptscriptstyle (1)}^{-j_{k}}$ 
or $y\in f_{\scriptscriptstyle (1)}^{-j_{k+1}}$ and $z\in F_{\scriptscriptstyle (1)}^{j_{k+1}}$.
If we assume that $|y, r(F_{\scriptscriptstyle (1)}^{j_{k}})|\geq \frac{1}{2}|y, z|$ we come back to the situation in point 2. Otherwise \newline
$|y, r(F_{\scriptscriptstyle (1)}^{j_{k}})|< \frac{1}{2}|y, z|$ then  $|r(F_{\scriptscriptstyle (1)}^{j_{k}}), z|\geq C_6  |F_{n}|$. Let $f_{\scriptscriptstyle (2)}^{-i}$ and $F_{\scriptscriptstyle (2)}^{i}$ be the preimage and the gap of $\F_{\scriptscriptstyle (2)}$ contained in $f_{\scriptscriptstyle (1)}^{-j_{k}}$. Then we can have two possibilities: either $z\in F_{\scriptscriptstyle (2)}^{i}$ or $z\in f_{\scriptscriptstyle (2)}^{-i}$. In the first case we use the same argument than in point 2. For the second case it is enough to observe that 
$F_{\scriptscriptstyle (2)}^{i}\subset (r(F_{\scriptscriptstyle (1)}^{j_{k}}), z)$, then as before:
\[
\frac{|\varphi(x),\varphi(w)|}{|\varphi(y),\varphi(z)|}\leq\frac{|G_{n}|}{|G_{\scriptscriptstyle (2)}^{i}|}\leq C_5
\]

\item $y,z\in f_{\scriptscriptstyle (1)}^{-j_{k}}$ or  $y,z\in f_{\scriptscriptstyle (1)}^{-j_{k+1}}$. By the hypothis condition $|y,z|\geq C |F_{n}|$, up to a finite number of steps, we may assume that $1$ is the biggest natural number such that $y,z\in f_{\scriptscriptstyle (1)}^{-j_{k}}$. Observe that, using the same notation in point 3, we may have two possibilities: either $y\in F_{\scriptscriptstyle (2)}^{i}$ and $z\in f_{\scriptscriptstyle (2)}^{-i}$ or $x,y\in F_{\scriptscriptstyle (2)}^{i}$. Now it is enough to repeat the arguments of points 2 and 3.

\end{enumerate}

\end{proof}

\begin{rem}\label{comnew2}
Observe that by Remark \ref{howfillpreimages} the properties found in Lemma \ref{key1} remain true also for the elements of the partition $\F_{\scriptscriptstyle (k)}$ for all finite natural number $k$.
\end{rem}

We are now ready to prove that the function $\varphi$ constructed in Proposition
\ref{extension} is a quasi-symmetric homeomorphism.
 
\begin{prop}\label{qscantor}
There exists a positive constant $C$ such that for all $x,y,z\in K_f$ we have
\begin{equation}
C^{-1}\leq\frac{|\varphi(x),\varphi(y)|}{|\varphi(y),\varphi(z)|}\leq C,
\end{equation}
provided that $|x,y|=|y,z|$. 
\end{prop}
\begin{proof}

Let $x,y,z\in K$ such that $|x,y|=|y,z|$. 
Let $n>0$ be the smallest natural number such that a primage $f^{-i_n}$ of $\F_n$ 
separates two of the three points $x$, $y$ and $z$. 
This means in particular that $x,y,z$ are contained in a same gap $F_{n-1}$ 
of the dynamical partition $\F_{n-1}$. 
Because of symmetry we may  suppose that $f^{-i_n}$ separates $y$ and $z$. 
So, by Remark \ref{ppropertiesofconstraction}, $g^{-i_n}\subset (\varphi(y), \varphi(z))$ and by Lemma \ref{bounpreim} there exists a positive constant $C_1$ such that  
\[
\frac{|\varphi(x),\varphi(y)|}{|\varphi(y),\varphi(z)|}\leq\frac{|G_{n-1}|}{|g^{-i_n}|}\leq C_1.
\]

For the other inequality observe by Lemma \ref{bounpreim} that, there exists a positive constant $C_2$ such that $|x,y|=|y,z|\geq |f^{-i_n}|\geq{C_2} |F_{n-1}|$. 
So, by Lemma \ref{key} there exists a positive constant $C_3$ such that
\[
\frac{|\varphi(y),\varphi(z)|}{|\varphi(x),\varphi(y)|}\leq{C_3}.
\]
The proof is now complete.
\end{proof}
\begin{rem}
Observe that Proposition \ref{qscantor} tells us that the conjugacy $\varphi_0$ constructed in Proposition \ref{Cantorhomeo} is quasi-symmetric.
\end{rem}

\begin{figure}

\begin{pspicture}(0,-1)(10,3)
   \psline[arrows={-)}](0,0)(2,0)
 \psdot(3.5,0)\uput{0.4}[-90](3.5,0){$x$}
 
  \uput{0.5}[-90](5,1.5){${F^{k+1}_{n}}$}
  
 \psdot(6,0)\uput{0.4}[-90](6,0){$y$}
 
\psline[arrows={(-)}](7,0)(10.2,0)
\uput{0.5}[-90](9,1.5){${f^{-k_n}}$}

 \psdot(9.3,0)\uput{0.4}[-90](9.3,0){$z$}
 
  
  \uput{0.5}[-90](12.3,1.5){${F^{k}_{n}}$}
  
   \psline[arrows={(-}](14,0)(15,0)
 
 \psline[linewidth=0.25pt]{-}(1,0)(15,0)

\end{pspicture}
\caption{}
\label{fig:primera}
\end{figure}

\begin{figure}

\begin{pspicture}(0,-1)(10,3)

  \psline[arrows={-}](6,0)(8,0)\uput{0.5}[-90](3,1.5){${F_{n}^{k+1}}$}
  
   \psdot(6,0)\uput{0.4}[-90](6,0){$x$}
  
 \psdot(9,0)\uput{0.4}[-90](9,0){$y$}
 
\psline[arrows={(-)}](4.6,0)(10,0)\uput{0.5}[-90](8,1.5){${f^{-k_n}}$}

%
%
%
 \psdot(12.3,0)\uput{0.4}[-90](12.3,0){$z$}
 
 \uput{0.5}[-90](12.3,1.5){${F^{k}_{n}}$}
  
   \psline[arrows={(-}](14,0)(15,0)
 
 \psline[linewidth=0.25pt]{-}(1,0)(15,0)

\end{pspicture}
\caption{}
\label{fig:seconda}
\end{figure}

\begin{figure}
\begin{pspicture}(0,-1)(10,3)
 \psdot(3,0)\uput{0.4}[-90](3,0){$x$}
 
  \psline[arrows={-}](6,0)(8,0)\uput{0.5}[-90](3,1.5){${F_{n}^{k+1}}$}
  
 \psdot(7,0)\uput{0.4}[-90](7,0){$y$}
 
\psline[arrows={(-)}](4.6,0)(10,0)\uput{0.5}[-90](8,1.5){${f^{-k_n}}$}

%
%
%
 \psdot(12.3,0)\uput{0.4}[-90](12.3,0){$z$}
 
\uput{0.5}[-90](12.3,1.5){${F^{k}_{n}}$}
  
   \psline[arrows={(-}](14,0)(15,0)
 
 \psline[linewidth=0.25pt]{-}(1,0)(15,0)
\end{pspicture}
\caption{}
\label{fig:terza}
\end{figure}

\begin{prop}\label{samegap}
There exists a positive constant $C$ such that for any gap $F_{n-1}$ of $\F_{n-1}$ and
any $x,y,z\in F_{n-1}$ we have
\begin{equation}
C^{-1}\leq\frac{|\varphi(x),\varphi(y)|}{|\varphi(y),\varphi(z)|}\leq C,
\end{equation}
provided that $|x,y|=|y,z|$ and $x,y,z$ do not all belong to the same
preimage of any subsequent generation. \end{prop}
\begin{proof}
Without loss of generality we may suppose that $n$ is the biggest natural number such that the points $x,y$ and $z$ 
belong to a  gap $F_{n-1}$ of $\F_{n-1}$.

The proof goes by cases:
\begin{enumerate}
\item Suppose that $x$ and $y$ belong to a gap $F^{k+1}_{n}$ of the dynamical partition $\F_n$ 
and that $z$ is in the adjacent right preimage $f^{-k_n}$.

 The reader may refer to Figure \ref{fig:primera}.\\

By Lemma \ref{moveFromPreimageRight} there exists $C>0$ and a gap $F_{m}$ of 
the dynamical partition $\F_m$, $m\geq n$ such that 
\begin{equation}\label{betterpreimage}
|F_{m}|\leq C|x,r(F^{k+1}_{n})|.
\end{equation}
 Let us consider the point $\tilde z\in F^{k+1}_{n}$ such that $|\tilde z, r(F^{k+1}_{n})|=|r(F^{k+1}_{n}), z|$. 
Because of $|x,y|=|y,z|$, we have  $\tilde z\in (x,r(F^{k+1}_{n}))\subset F_m$. 

Further
\begin{equation}\label{inequality1}
|F_m|\leq 2C |x,y|
\end{equation}
and
\[
\frac{|\varphi(y),\varphi(z)|}{|\varphi(x),\varphi(y)|}\leq\frac{2\max(|\varphi(y),\varphi(r(G^{k+1}_{n}))|,|\varphi(\tilde{z}), \varphi(r(G^{k+1}_{n}))|)}{|\varphi(x),\varphi(y)|}\leq 2C_1. 
\]

where the constant $C_1$ comes from Lemma \ref{key} if $(x,y)$ is not contained in a preimage of a subsequent generation and from Lemma \ref{key1} otherwise.

For the other inequality, 
\begin{enumerate}
\item If $|r(F^{k+1}_{n}),z|\geq\frac{1}{2}|y,z|$, then 
\begin{eqnarray}\label{frequentinequality}
|F_m|&\leq &C |x,r(F^{k+1}_{n})|\leq C 2 |y,z|\\&\leq& C4 |r(F^{k+1}_{n}),z|=C4 |\tilde z, r(F^{k+1}_{n})|.
\end{eqnarray}

\item If $|r(F^{k+1}_{n}),z|<\frac{1}{2}|y,z|$ then
\[
|F_m|\leq C |x,r(F^{k+1}_{n})|=C(|x,y|+|y,r(F^{k+1}_{n})|)\leq C 3 |y,r(F^{k+1}_{n})|.
\]
\end{enumerate}
In both the cases we conclude as before using Lemma \ref{key}.

\item Suppose that $x,y$ belong to a preimage $f^{-k_n}$ of the dynamical partition $\F_n$ and that $z$ is in the adjacent right gap $F^{k}_{n}$.
The reader may refer to Figure \ref{fig:seconda}.

 Recall that by the construction of $\varphi$, $f^{-k_n}$ 
is divided into four sub-sets $F_{\scriptscriptstyle (1)}^{j_{k_1}}$, $f_{\scriptscriptstyle (1)}^{-j_{k_1}}$, 
$f_{\scriptscriptstyle (1)}^{-j_{k_1+1}}$ and $F_{\scriptscriptstyle (1)}^{j_{k_1+1}}$ (see Remark \ref{howfillpreimages}) all comparable with $f^{-k_n}$.

\begin{enumerate}
\item Suppose now that $x,y\in F_{\scriptscriptstyle (1)}^{j_{k_1}}\cup f_{\scriptscriptstyle (1)}^{-j_{k_1}}\cup  f_{\scriptscriptstyle (1)}^{-j_{k_1+1}} $. Notice that $|x,y|=|y,z|\geq | F_{\scriptscriptstyle (1)}^{j_{k_1+1}}|\geq C_2 |F_{n-1} |$, so by Lemma \ref{key1} and by Remark \ref{comnew2} there exists a positive constant $C_3$ such that 

\begin{eqnarray*}
\frac{|\varphi(y),\varphi(z)|}{|\varphi(x),\varphi(y)|}&\leq C_3
\end{eqnarray*}

For the other inequality it is enough to observe that $F_{\scriptscriptstyle (1)}^{j_{k_1+1}}\subset (y,z)$.

Observe that the same kind of argument holds also for $x\in F_{\scriptscriptstyle (1)}^{j_{k_1}}\cup f_{\scriptscriptstyle (1)}^{-j_{k_1}}$ and $y\in  F_{\scriptscriptstyle (1)}^{j_{k_1+1}}$ (in fact in this case $f_{\scriptscriptstyle (1)}^{-j_{k_1+1}}\subset (x,y)$).

\item Suppose now that $x,y\in F_{\scriptscriptstyle (1)}^{j_{k_1+1}}$. By Lemma \ref{moveFromPreimageRight} and Remark \ref{comnew} there exists $C_4>0$ and a gap $F_{\scriptscriptstyle (1)}^{j_{k_2}}$ of 
the dynamical partition $\F_{\scriptscriptstyle (1),k_2}$, $k_2\geq k_1+1$ such that 
\begin{equation*}
|F_{\scriptscriptstyle (1)}^{j_{k_2}}|\leq C_4|x,r(F_{\scriptscriptstyle (1)}^{j_{k_2}})|.
\end{equation*}
Suppose now that $|y,l(F^{k}_{n})|\geq \frac{1}{2}|y,z|$ (the case $ |y,l(F^{k}_{n})|< \frac{1}{2}|y,z|$ is analogous). So, $|F_{\scriptscriptstyle (1)}^{j_{k_2}}|\leq C_4|x,r(F_{\scriptscriptstyle (1),k_1+1}^{-j})|\leq C_4|x,y|+|y,l(F^{k}_{n})|\leq 2C_4 |y,l(F^{k}_{n})|$. As in point 1, by Lemma \ref{key} and Remark \ref{comnew} there exists a positive constant $C_5$ such that
\[
\frac{|\varphi(x),\varphi(y)|}{|\varphi(y),\varphi(z)|}
\leq C_5
\]
For the other inequality we notice that $|F_{\scriptscriptstyle (1)}^{j_{k_2}}|\leq C_4|x,l(F^{k}_{n})|=C_4(|x,y|+|y,l(F^{k}_{n})|)\leq 2C_4|x,y|$. We conclude applying Lemma \ref{key} and Remark \ref{comnew} if  $(x,y)$ are not contained in a preimage of some next generation or  Lemma \ref{key1} and Remark \ref{comnew2} otherwise.
\item Assume that $x\in F_{\scriptscriptstyle (1)}^{j_{k_1}}\cup f_{\scriptscriptstyle (1)}^{-j_{k_1}}\cup  f_{\scriptscriptstyle (1)}^{-j_{k_1+1}} $ and $y\in F_{\scriptscriptstyle (1)}^{j_{k_1+1}}$. Observe that $|x,y|=\frac{1}{2}(|x,y|+|y,z|)\geq \frac{1}{2}  |F_{\scriptscriptstyle (1)}^{j_{k_1+1}}|\geq \frac{C_6}{2}|F_{n-1}| $. By Lemma \ref{key1},

\begin{eqnarray*}
\frac{|\varphi(y),\varphi(z)|}{|\varphi(x),\varphi(y)|}&\leq C_7
\end{eqnarray*}
For the other inequalty we use the same argument and Lemma \ref{key}.
\end{enumerate}

\item Suppose now that $x\in F^{k+1}_{n}$, $y\in f^{-k_n}$ and $z\in F^k_{n}$. The reader may refer to Figure \ref{fig:terza}.  As below we use notations by Remark \ref{howfillpreimages}.

We can suppose that $y\in F_{\scriptscriptstyle (1)}^{j_{k_1}}\cup f_{\scriptscriptstyle (1)}^{-j_{k_1}} $. Because of symmetry the case $y\in f_{\scriptscriptstyle (1)}^{-j_{k_1+1}}\cup F_{\scriptscriptstyle (1)}^{j_{k_1+1}}$ is completely analogous. 

Observe that $(y,z)$ contains  $f_{\scriptscriptstyle (1)}^{-j_{k_1+1}}\cup F_{\scriptscriptstyle (1)}^{j_{k_1+1}}$, then by Lemma \ref{key}

\[
\frac{|\varphi(x),\varphi(y)|}{|\varphi(y),\varphi(z)|}\leq C_8.
\]
For the other inequality we suppose that $|l(f^{-k_n}), y|\geq \frac{1}{2}|x,y|$. Then $|l(f^{-k_n}), y|\geq \frac{1}{2}|x, y|=\frac{1}{2}|y, z|\geq \frac{1}{2} |f_{\scriptscriptstyle (1)}^{-j_{k_1+1}}\cup F_{\scriptscriptstyle (1)}^{j_{k_1+1}}|\geq \frac{C_9}{2}|F_{n-1}|$. As before the conclusion comes from Lemma \ref{key1}. 

The last case to be considered is $|l(f^{-k_n}), y|< \frac{1}{2}|x, y|$. In this situation $|x,l(f^{-k_n}) |> \frac{1}{2}|x, y|=\frac{1}{2}|y, z|\geq \frac{1}{2}|f_{\scriptscriptstyle (1)}^{-j_{k_1+1}}\cup F_{\scriptscriptstyle (1)}^{j_{k_1+1}}|\geq \frac{C_7}{2}|F_{n-1}|$. By Lemma \ref{key} also this case is settled.

 \item Suppose now that $x\in f^{-k}_{n}$, $y\in F^k_{n}$ and $z\in f^{-k+1}_{n}$. By Lemma \ref{moveFromPreimageRight} there exists $C_{10}>0$ and a gap $F_{m}$ of 
the dynamical partition $\F_m$, $m\geq n$ such that 
\begin{equation}\label{FM}
|F_{m}|\leq C_{10}|y,l(f^{-k+1}_{n})|.
\end{equation}
Moreover we recall that by Remark \ref{howfillpreimages}, $ f^{-k}_{n}$ is divided into four subsets $F_{\scriptscriptstyle (1)}^{k_{k_1}}$, $f_{\scriptscriptstyle (1)}^{-k_{k_1}}$, 
$f_{\scriptscriptstyle (1)}^{-k_{k_1+1}}$, $F_{\scriptscriptstyle (1)}^{k_{k_1+1}}$ and $f^{-k+1}_{n}$ into $F_{\scriptscriptstyle (1)}^{k+1_{k_1}}$, $f_{\scriptscriptstyle (1)}^{-k+1_{k_1}}$, 
$f_{\scriptscriptstyle (1)}^{-k+1_{k_1+1}}$, $F_{\scriptscriptstyle (1)}^{k+1_{k_1+1}}$. 
 \begin{itemize}
 \item [(a)]Suppose that $x\in F_{\scriptscriptstyle (1)}^{k_{k_1}}\cup f_{\scriptscriptstyle (1)}^{-k_{k_1}}\cup f_{\scriptscriptstyle (1)}^{-k_{k_1+1}}$. In other worlds, $(x,y)$ contains $ F_{\scriptscriptstyle (1)}^{k_{k_1+1}}$. We conclude as in the other cases using Lemma \ref{key}.
%
%
%
%
%

\item [(b)] Suppose now that $x\in  F_{\scriptscriptstyle (1)}^{k_{k_1+1}}$. If $z\in f_{\scriptscriptstyle (1)}^{-k+1_{k_1}}\cup f_{\scriptscriptstyle (1)}^{-k+1_{k_1+1}}\cup  F_{\scriptscriptstyle (1)}^{k_{k_1+1}}$ because of symmetry the situation has already been analyzed in the previous point. We can then suppose that $z\in F_{\scriptscriptstyle (1)}^{k+1_{k_1}}$. If $|r( f^{-k}_{n}), y|\geq \frac{1}{2}|x,y|$, then $4|r( f^{-k}_{n}), y|\geq |x,y|+|y,z|\geq |F_n^k|$. By Lemma \ref{key}
\[
\frac{|\varphi(y),\varphi(z)|}{|\varphi(x),\varphi(y)|}\leq C_{11}
\]
The case $|r( f^{-k}_{n}), y|< \frac{1}{2}|x,y|$ is analogous. 

For the other inequality, because of symmetry the procedure is exactly the same using $ l(f^{-k+1}_{n})$ at the place of $r( f^{-k}_{n})$.
 \end{itemize}
\item Suppose that or $x,y$ or $y,z$ are separated by at least one preimage or one gap of the partition $n$ which will be denoted by $F_n$. Without loss of generality we can suppose that $F_n$ separate $x$ and $y$, then
\[
\frac{|\varphi(y),\varphi(z)|}{|\varphi(x),\varphi(y)|}\leq\frac{|G_{n-1}|}{|G_n|}\leq C_{12}.
\]
For the other inequality we observe that $|y,z|=|x,y|\geq |F_n|\geq C |F_{n-1}|$ and we apply Lemma \ref{key}.

\end{enumerate}

\end{proof}
\begin{prop}\label{samepreimage}
There exists a positive constant $C$ such that for any preimage $f^{-n}$ and
any $x,y,z\in f^{-n}$ we have
\begin{equation}
C^{-1}\leq\frac{|\varphi(x),\varphi(y)|}{|\varphi(y),\varphi(z)|}\leq C,
\end{equation}
provided that $|x,y|=|y,z|$.
\end{prop}
\begin{proof}
Without loss of generality we may assume that, for all $n>0$ $x,y,z$ are not all contained in a some preimage of the partition  $\F_{\scriptscriptstyle (n)}$.

Following Remark \ref{howfillpreimages} we consider the subdivision of $f^{-n}$ into four subsets $F_{\scriptscriptstyle (1)}^{k_{k_1}}$, $f_{\scriptscriptstyle (1)}^{-k_{k_1}}$, 
$f_{\scriptscriptstyle (1)}^{-k_{k_1+1}}$, $F_{\scriptscriptstyle (1)}^{k_{k_1+1}}$.


\begin{enumerate}
\item Suppose that $x,y,z$ are all in $F_{\scriptscriptstyle (1)}^{k_{k_1}}\cup f_{\scriptscriptstyle (1)}^{-k_{k_1}}$, (because of symmetry the same arguments apply for the case $x,y,z\in f_{\scriptscriptstyle (1)}^{-k_{k_1+1}}\cup F_{\scriptscriptstyle (1)}^{k_{k_1+1}})$.

\begin{itemize}
\item[(a)]If $x,y\in F_{\scriptscriptstyle (1)}^{k_{k_1}}$ and $z\in f_{\scriptscriptstyle (1)}^{-k_{k_1}}$ we repeat the arguments of point 1 of Proposition \ref{samegap}.
\item[(b)] If $x\in F_{\scriptscriptstyle (1)}^{k_{k_1}}$ and $y,z\in f_{\scriptscriptstyle (1)}^{-k_{k_1}}$ we repeat the arguments of point 2 of Proposition \ref{samegap}. 
\item[(c)] If $x,y,z\in F_{\scriptscriptstyle (1)}^{k_{k_1}}$ we apply Proposition \ref{samegap} .
\end{itemize}

\item Suppose $x,y$ are in $F_{\scriptscriptstyle (1)}^{k_{k_1}}\cup f_{\scriptscriptstyle (1)}^{-k_{k_1}}$ and $z\in f_{\scriptscriptstyle (1)}^{-k_{k_1+1}}\cup F_{\scriptscriptstyle (1)}^{k_{k_1+1}}$ (because of symmetry the same argument can be used for $x\in F_{\scriptscriptstyle (1)}^{k_{k_1}}\cup f_{\scriptscriptstyle (1)}^{-k_{k_1}}$ and  $y,z\in f_{\scriptscriptstyle (1)}^{-k_{k_1+1}}\cup F_{\scriptscriptstyle (1)}^{k_{k_1+1}}$)

\begin{itemize}
\item [(a)] If $x,y$ are in $F_{\scriptscriptstyle (1)}^{k_{k_1}}\cup f_{\scriptscriptstyle (1)}^{-k_{k_1}}$ and $z\in F_{\scriptscriptstyle (1)}^{k_{k_1+1}}$, then 
\[
\frac{|\varphi(x),\varphi(y)|}{|\varphi(y),\varphi(z)|}\leq\frac{|g^{-n}|}{| g_{\scriptscriptstyle (1)}^{-k_{k_1+1}}|}\leq C_1
\]
For the other inequality we observe that $|x,y|=|y,z|\geq |  f_{\scriptscriptstyle (1)}^{-k_{k_1+1}}|\geq C_2 |F_{\scriptscriptstyle (1)}^{k_{k_1}}\cup f_{\scriptscriptstyle (1)}^{-k_{k_1}}|$. We conclude using Lemma \ref{key1} and Remark \ref{comnew2}.

\item [(b)] If $x,y$ are in $F_{\scriptscriptstyle (1)}^{k_{k_1}}$ and $z\in  f_{\scriptscriptstyle (1)}^{-k_{k_1+1}}$ then we use exactly the same arguments as in the previous point. 
In this case the key point is the fact that $(y,z)$ contains $ f_{\scriptscriptstyle (1)}^{-k_{k_1}}$.
\item [(c)] If $x\in F_{\scriptscriptstyle (1)}^{k_{k_1}}$, $y\in f_{\scriptscriptstyle (1)}^{-k_{k_1}}$ and $z\in  f_{\scriptscriptstyle (1)}^{-k_{k_1+1}}$. We observe that $|y,z|=|x,y|=\frac{1}{2}(|x,y|+|y,z|)\geq\frac{1}{2}|f_{\scriptscriptstyle (1)}^{-k_{k_1}}|$. For both inequalities apply Lemma \ref{key1} and Remark \ref{comnew2}. 

%

\item [(d)]If $x,y\in f_{\scriptscriptstyle (1)}^{-k_{k_1}}$ and $z\in f_{\scriptscriptstyle (1)}^{-k_{k_1+1}}$. Without loss of generality we may assume that  $f_{\scriptscriptstyle (1)}^{-k_{k_1}}$ is the bigger right extreme preimage which contain $(x,y)$. Let $f_{\scriptscriptstyle (2)}^{-i}$ and $F_{\scriptscriptstyle (2)}^{i}$ be the preimage and the gap of $\F_{\scriptscriptstyle (2)}$ contained in $ f_{\scriptscriptstyle (1)}^{-k_{k_1}}$ and comparable with it. If $x,y\in F_{\scriptscriptstyle (2)}^{i} $ we use exactly the same argument as before ($f_{\scriptscriptstyle (2)}^{-i}\subset (y,z)$). We can then suppose that $x\in F_{\scriptscriptstyle (2)}^{i} $ and $y\in f_{\scriptscriptstyle (m+2)}^{-i}$.  Let $f_{\scriptscriptstyle (3)}^{-j}$ and $F_{\scriptscriptstyle (3)}^{j}$ be the preimage and the gap of $\F_{\scriptscriptstyle (3)}$ contained in $ f_{\scriptscriptstyle (2)}^{-i}$ and comparable with it. If $y\in f_{\scriptscriptstyle (3)}^{-j}$ then $F_{\scriptscriptstyle (3)}^{j}\subset (x,y)$ and if $y\in F_{\scriptscriptstyle (3)}^{j}$ then $f_{\scriptscriptstyle (3)}^{-j}\subset (y,z)$. In both the cases we use exactly the same arguments as before.

\end{itemize}
\end{enumerate}

\end{proof}

Finally we can state the main result of the section.
\begin{prop}

The circle homeomorphism $\varphi$ constructed in Proposition \ref{extension} is quasi-symmetric.
\end{prop}
\begin{proof}
Apply Proposition \ref{samepreimage} and Proposition \ref{samegap}.

\end{proof}

\section{Obstacle in the construction of a global \\ quasi-symmetric conjugacy}\label{noQS}
Observe that the quasi-symmetric homeomorphism $\varphi$ constructed in Theorem \ref{maintheo} conjugate $f$ and $g$ on $K_f\cup U_f$, 
i.e. $\varphi_{|K_f\cup U_f}\circ f=g\circ\varphi_{|K_g\cup U_g} $. 
 The aim of this section is to explain the difficulty to extend $\varphi$ to a quasi-symmetric conjugacy on the preimages of the flat interval. 

Let $f,g\in\Lsec$ with the same rotation number of bounded type. Observe that with a similar procedure used for proving Theorem \ref{maintheo} we can construct an homeomorphism $\varphi: K_f\cup U_f\to K_g\cup U_g$ which conjugate $f$ and $g$ on $K_f\cup U_f$. The most natural way to extend it to all the preimages of the flat interval is the following:
if $x\in f^{-n}$, then $\varphi(x)= (g^{-n}\circ\varphi\circ f^{n})(x)$. We get then a global conjugacy of $f$ and $g$ which is quasi-symmetric on $K_f\cup U_f$.


In order to prove that $\varphi$ is quasi-symmetric also on the preimages of the flat interval the first step would be to establish that for all $n$, the transition functions $f^n$ and $g^n$ are quasi-symmetric. Unfortunately for function in $\Lsec$ this is not true.

\begin{prop}
Let $n\in \N$ and let $P=f^{-i}$ be a preimage of the flat interval belonging to the dynamical partition $\F_n$. If $n$ is big enough, $f^{i}_{|P}$ is not quasi-symmetric.
\end{prop}
\begin{proof}
Without loss of generality we can suppose that $i=q_n$. Let us consider two subsets, $A$ and $B$ of $U_f$ such that $l(A)=l(U_f)$, $r(A)=r(U_f)$. We also assume that $|A|$ and $|B|$ are comparable with $|f^{-q_n}|$.

Observe that 
\begin{eqnarray*}
\frac{|f^{-q_n}(A)|}{|f^{-2q_n}, r(f^{-q_n}(A)) |}&\geq& \frac{|f^{-2q_n}|}{|l(f^{-2q_n}), f^{-q_n}(A) |}\frac{|f^{-q_n}(A)|}{|f^{-2q_n}(A), r(f^{-q_n}(A)) |}\\&\geq & C_1 \frac{|f^{-q_n}|}{|l(f^{-q_n}), A) |}\frac{|A|}{|f^{-q_n}(A), r(A) |}\geq C_1 C_2 C_3
\end{eqnarray*}
where the constant $C_1$ comes from the expanding properties of the cross-ratio $Cr$ after $i$ iterate, $C_2$ by Proposition \ref{figo1} and $C_3$ by the hypothesis of comparability of $|A|$ and $|f^{-q_n}|$.

So, by the previous inequality and by Lemma \ref{bboungap}, there exist two constants $C_5$ and $C_6$ such that 
\[
|f^{-q_n}(A)|\geq C_5|f^{-2q_n}, r(f^{-q_n}(A))|\geq C_5C_6 |P|.
\]
In the same way we can prove that there exists a positive constant $C_7$ such that
\[
|f^{-q_n}(B)|\geq C_7 |P|.
\]
Let us consider now the following three points: $z=l(f^{-q_n}(B))$, $y=r(f^{-q_n}(A))$ and $x\in f^{-q_n}(A)$ such that $|x,y|=|y,z|$. Suppose that $f^{q_n}$ is quasi-symmetric then there exists a constant $C$ such that 
\[
C^{-1}\leq\frac{|f^{q_n}(x),f^{q_n}(y)|}{|f^{q_n}(y),f^{q_n}(z)|}\leq C.
\]
  In particular $|A|$ is comparable with $|U_f\setminus (A\cup B)|$ and this is not true for $n$ big (recall that we are assuming that $|A|$ is comparable with $|f^{-q_n}|$).  
\end{proof}
In conclusion, in order to have a global quasi-symmetric conjugacy we probably need some more hypothesis and a deeper control of the distortion properties of the transitions functions $f^n$. 

\section*{Acknowledgments}
I am very grateful to Prof. G. \'Swi\c atek for introduction to the topic and for many helpful discussions. 

\section{Appendix}
In this section we comment that our construction solve a more general analytic problem concerning the extension of quasi-symmetric homeomorphisms of Cantor sets.

Let $\{C_n\}_{n\in\N}$ be a bounded sequence of positive integers and let for all $n$, $\{I_{i_n}\}_{0\leq i_n\leq C_n}$ and $\{I'_{i_n}\}_{0\leq i_n\leq C_n}$ 
be two sequences of subintervals of $\S$. 
We suppose that for all $n\in\N$ two consecutive subinterval $I_{i_n}$ and $I_{i_n+1}$ (resp $I'_{i_n}$ and $I'_{i_n+1}$) 
are uniformly comparable and we consider the following Cantor sets
\[
K=\bigcap_{n\geq1}\left(\bigcup_{0\leq i_n\leq C_n}I_{i_n}\right),
\]

\[
K'=\bigcap_{n\geq1}\left(\bigcup_{0\leq i_n\leq C_n}I'_{i_n}\right)
\]
Using the proof of Theorem of \ref{maintheo} we deduce the following property:
\begin{theo}
There exists a quasi-symmetric homeomorphism $\varphi_0:K\to K'$ which can be extended to a circle quasi-symmetric homeomorphism.
\end{theo}

\end{document}